\documentclass[12pt,reqno]{amsart}
\usepackage{amsthm,amsfonts,amsmath,amssymb,bookmark,appendix}

\usepackage{amsmath,amsfonts,amsthm,amssymb,amsxtra}
\usepackage{amsxtra, amssymb, mathrsfs}
\usepackage[usenames,dvipsnames]{color}

\newtheorem{theorem}{Theorem}[section]

\newtheorem{lemma}[theorem]{Lemma}
\newtheorem{corollary}[theorem]{Corollary}

\theoremstyle{definition}

\theoremstyle{remark}
\newtheorem{remark}[theorem]{\bf Remark}

\numberwithin{equation}{section}

\newcommand{\C}{\mathbb{C}}

\newcommand{\R}{\mathbb{R}}
\newcommand{\Sp}{\mathbb{S}}

\newcommand{\Z}{\mathbb{Z}}

\begin{document}

\title[Improved decay for Schr\"odinger flows]{Improved time-decay for a class of scaling critical electromagnetic Schr\"odinger flows}

\author{Luca Fanelli}
\address{Luca Fanelli: Dipartimento di Matematica "Guido Castelnuovo", SAPIENZA Universit$\grave{\text{a}}$ di Roma, P.le A. Moro 5, 00185 Roma, Italy}
\email{fanelli@mat.uniroma1.it}

\author{Gabriele Grillo}
\address{Gabriele Grillo: Dipartimento di Matematica, Politecnico di Milano, Piazza Leonardo da Vinci 32, 20133 Milano, Italy}
\email{gabriele.grillo@polimi.it}

\author{Hynek Kova\v r\'{\i}k}
\address{Hynek Kova\v r\'{\i}k DICATAM, Sezione di Matematica, Universita` degli studi di Brescia, Via Branze, 38, 25123 Brescia, Italy}
\email{hynek.kovarik@ing.unibs.it}


\thanks{
The first author was supported by the project FIRB 2012:
``Dispersive dynamics: Fourier Analysis and Variational Methods'' funded by MIUR. G.~G. and H.~K. were supported by the Gruppo Nazionale per Analisi Matematica, la Probabilit\`a e le loro Applicazioni (GNAMPA) of the Istituto Nazionale di Alta Matematica (INdAM). G.~G. acknowledges support of MIUR-PRIN2012 grant for the project  ``Equazioni alle derivate parziali di tipo ellittico e parabolico: aspetti geometrici, disuguaglianze collegate, e applicazioni''. The support of MIUR-PRIN2010-11 grant for the project  ``Calcolo delle variazioni'' (H.~K.) is also gratefully acknowledged.
}

\subjclass[2000]{35J10, 35L05.}
\keywords{Schr\"odinger equation, electromagnetic potentials, representation
formulas, decay estimates}

\begin{abstract}
  We consider a Schr\"odinger hamiltonian $H(A,a)$ with scaling critical and time independent external electromagnetic potential, and assume that the angular operator $L$ associated to $H$ is positive definite. We prove the following: if $\|e^{-itH(A,a)}\|_{L^1\to L^\infty}\lesssim t^{-n/2}$, then
 $ \||x|^{-g(n)}e^{-itH(A,a)}|x|^{-g(n)}\|_{L^1\to L^\infty}\lesssim t^{-n/2-g(n)}$, $g(n)$ being a positive number, explicitly depending on the ground level of $L$ and the space dimension $n$. We prove similar results also for the heat semi-group generated by $H(A,a)$.
\end{abstract}

\date{\today}

\maketitle


\section{\bf Introduction}\label{sec:intro}
This note is concerned with the dispersive properties of electromagnetic Schr\"odinger groups of the form $e^{-itH}$, where
\begin{equation}\label{eq:H}
  H(A,a) = \left(-i \nabla +|x|^{-1}\ A\Big(\frac{x}{|x|}\Big)\right)^2 +|x|^{-2}\ a\Big(\frac{x}{|x|}\Big)
\end{equation}
acts on a suitable form domain in $L^2(\R^n)$ with $n\geq 2$. Here $A\in W^{1,\infty}(\Sp^{n-1}, \R^n)$ and $a\in W^{1,\infty}(\Sp^{n-1}, \R)$, where $\Sp^{n-1}$ denotes the $n-$dimensional unit sphere. We also assume that $A$ satisfies the transversal gauge condition
\begin{equation}\label{eq:transversal}
A(\theta) \cdot \theta = 0 \qquad \forall\  \theta\in\Sp^{n-1}.
\end{equation}
Once it is well defined, the flow $e^{-itH}$ defines the unique solution $\psi(t,x):=e^{-itH}f(x)$ to the Schr\"odinger equation
\begin{equation}\label{eq:schro}
  \partial_t\psi(t,x)=-iH(A,a)\psi,
  \qquad
  \psi(0,x)=f(x)\in L^2(\R^n).
\end{equation}
The wavefunction $\psi(t,x):\R^{1+n}\to\C$
standardly describes the interaction of a free (non-relativistic) particle with an external time-independent electromagnetic field $(V,B)$
\begin{equation*}
   V(x):=\nabla\left(|x|^{-2}a\left(\frac{x}{|x|}\right)\right),
   \quad
   B(x)=D\left(|x|^{-1}A\left(\frac{x}{|x|}\right)\right)-D\left(|x|^{-1}A\left(\frac{x}{|x|}\right)\right)^t.
\end{equation*}
Notice that the potentials in \eqref{eq:H} are such that equation \eqref{eq:schro} enjoys the same scaling invariance as the free Schr\"odinger equation, namely if $\psi$ solves \eqref{eq:schro}, then
\begin{equation*}
  \psi_\lambda(t,x):=\psi\left(\frac{t}{\lambda},\frac{x}{\lambda^2}\right).
\end{equation*}
solves the same equation, with scaled datum $f_\lambda(x)=f(x/\lambda^2)$.
For those reasons, we refer to \eqref{eq:H} as to a scaling critical electromagnetic hamiltonian.
An important role in the description of $H(A,a)$ is played by the angular hamiltonian
\begin{equation} \label{L}
L= \left(-i \nabla_{\Sp^{n-1}} +A\right)^2 +a \qquad \text{in} \quad L^2(\Sp^{n-1}),
\end{equation}
where $\nabla_{\Sp^{n-1}}$ denotes the spherical gradient on $\Sp^{n-1}$. In the free case $A\equiv a\equiv0$, we have $H(A,a)=-\Delta$ and $L$ coincides with the Laplace-Beltrami operator $L=\Delta_{\Sp^{n-1}}$ on $L^2(\Sp^{n-1})$.
Since $L$ is the inverse of a compact operator, its spectrum is purely discrete and consists of a sequence of eigenvalues $\mu_1(A,a) \leq \mu_2(A,a) \leq \dots $ accumulating at infinity. We denote by $\psi_k$ the corresponding sequence of normalized $L^2$-eigenfunctions:
\begin{equation} \label{psi-k}
L \, \psi_k = \mu_k(A,a)\, \psi_k, \qquad \|\psi_k\|_{L^2(\Sp^{n-1})}= 1.
\end{equation}
By the usual Hardy and diamagnetic inequalities, it is standard to show that the condition
\begin{equation}\label{eq:hardy}
  \mu_1(A,a)>\frac{(n-2)^2}{4}
\end{equation}
implies that the quadratic form associated to $H$ is positive definite, and this allows to work with the Friedrichs extension of $H$ and define the flow $e^{-itH}$ by Spectral Theorem.

In the recent years, a big effort has been spent in order to understand dispersive properties of scaling critical Schr\"odinger flows of the form $e^{-itH}$, with $H$ like in \eqref{eq:H}. The very first results, concerning with Strichartz estimates, have been obtained in the magnetic free case in \cite{BPSTZ1, BPSTZ}. The main idea was to prove some Morawetz-type estimates, based on uniform resolvent estimates, and then obtaining Strichartz estimates by a $TT^\star$-argument. This argument does not seem to work in presence of a scaling critical magnetic potential, and several results have been later obtained for sub-critical potentials, using the same kind of tricks (see e.g. \cite{BG, DF3, DFVV, EGS1, EGS2, GST, PSTZ, RS} and the references therein). Very recently, in \cite{f3p-1}, a quite explicit representation formula for $e^{-itH}$ has been provided, involving a convolution kernel written in terms of eigenvalues and eigenfunctions of $L$, for any critical hamiltonian as in \eqref{eq:H}. By suitably estimating this kernel, it has been possible to prove that the time-decay estimate
\begin{equation}\label{eq:disp}
   \left\|e^{-itH(A,a)}\right\|_{L^1(\R^n)\to L^\infty(\R^n)}\leq C\, t^{-\frac n2}
\end{equation}
holds, for any $t>0$, and some constant $C>0$, in some specific situations. The state of the art is the following: estimate \eqref{eq:disp} holds
\begin{itemize}
\item in dimension $n=3$ for the positive inverse-square electric potential ($A\equiv0$, $0\leq a=\text{constant}$ in \eqref{eq:H} (see \cite{f3p-1})
\item in dimension $n=2$ for any reasonable $A,a$ (see \cite{f3p-2}).
\end{itemize}
At the same time of \cite{f3p-1}, in \cite{gk}, an analogous formula has been independently obtained for $e^{-itH}$ in the specific case of the Aharonov-Bohm potential, i.e. $n=2$, $a\equiv0$, and $A=\alpha(-x_2/|x|,x_1/|x|)$, $\alpha\in\R$ in \eqref{eq:H}. One of the main results in \cite{gk} is the following: in suitable weighted topologies, one can observe a polynomial improvement in the decay rate $t^{-n(2}$; more precisely, the following estimate holds
\begin{equation} \label{eq:GK}
\| |x|^{-g(\alpha)}\ e^{-i t H(A,a)}\, |x|^{-g(\alpha)} \|_{L^1(\R^n)\to L^\infty(\R^n)} \ \leq \ C\ t^{-\frac n2-g(\alpha)},
\end{equation}
where $g(\alpha)=\text{dist}(\alpha,\Z)=\sqrt{\mu_1(A)}$, being $\mu_1$ the first eigenvalue of the angular operator $L$ as above.
The results of \cite{gk} where then extended to a wide class of regular magnetic fields with finite flux in \cite{ko2}.

\smallskip

The main ingredient in the proof of \eqref{eq:GK}, apart from the representation formula, is estimate \eqref{eq:disp}, which for the Aharonov-Bohm potential has been proven in \cite{f3p-1}.
The aim of this paper is to prove that the polynomial improvement in \eqref{eq:GK} is a generic fact, depending on the dimension and the ground level $\mu_1(A,a)$, occurring as a consequence of the dispersive estimate \eqref{eq:disp} holds.

\smallskip

\noindent We are now ready to state the main result of this paper.

\begin{theorem} \label{thm-main}
Let $n\geq 2$, $H(A,a)$ as in \eqref{eq:H}, where $A\in W^{1,\infty}(\Sp^{n-1}, \R^n)$, $a\in W^{1,\infty}(\Sp^{n-1}, \R)$ and $A$ satisfies \eqref{eq:transversal}. Let $L$ be as in \eqref{L},
denote by $\mu_1(A,a)$ its smallest eigenvalue, and assume that $\mu_1(A,a)\geq 0$. In addition, denote by
\begin{equation} \label{gn}
g(n) = \sqrt{\left(\frac{n-2}{2}\right)^2 + \mu_1(A,a)}\ -\frac{n-2}{2}.
\end{equation}
If, for all $t>0$ and some $C_0$, the following estimate holds
\begin{equation} \label{asume}
\| e^{-i t H(A,a)}\|_{L^1(\R^n)\to L^\infty(\R^n)} \ \leq \ C_0\ t^{-\frac n2},
\end{equation}
then there exists a constant $C$ such that
\begin{equation} \label{eq-main}
\big\| |x|^{-g(n)}\ e^{-i t H(A,a)}\, |x|^{-g(n)} \big\|_{L^1(\R^n)\to L^\infty(\R^n)} \ \leq \ C\ t^{-\frac n2-g(n)}
\end{equation}
holds, for all $t>0$.
\end{theorem}

\begin{remark}\label{rem:interpol}
Notice that, by interpolation between inequalities \eqref{asume} and \eqref{eq-main}, it follows that
\begin{equation}
\big \| |x|^{-\theta g(n)}\ e^{-i t H(A,a)}\, |x|^{-\theta g(n)} \big \|_{L^1(\R^n)\to L^\infty(\R^n)} \ \leq \ C_\theta\ t^{-\frac n2-\theta g(n)}
\end{equation}
holds for all $\theta\in [0,1]$, all $t>0$ and some $C_\theta$.
\end{remark}

\begin{remark}\label{rem:cit}
It well known that, in presence of a magnetic field, one always has $\mu_1(A,a)\geq\mu_1(0,a)$ (by the diamagnetic inequality). Estimate \eqref{eq-main} improves the decay rate of \eqref{asume} as soon as $\mu_1(A,a)>0$, condition which can be standardly rephrased by a variational characterization of $\mu_1$. This means that the more elliptic the angular hamiltonian is, the better the decay is expected to be, in suitable topologies. The validity of \eqref{eq-main} relies in fact on the validity of \eqref{asume}, which is still not known in a generic setting, and will be subject of further investigation in the future.
  Notice that estimate \eqref{eq-main} generalizes \eqref{eq:GK}. For Schr\"odinger evolutions, as far as we know, those are the first results in this direction.

On the other hand, this kind of phenomenon is qualitatively well known for the heat equation. For radial magnetic fields this was observed in \cite{ko1} in the $L^1 \to L^\infty$ setting. Later on, in \cite{kr} it was shown that similar effects occur, in a setting of weighted $L^2$ spaces, for more general class of magnetic fields. We will briefly discuss the properties of the semi-group generated by $H(A,a)$ in Section \ref{sec-heat}.
\end{remark}

\noindent In view of Theorem \ref{thm-main} and the results in \cite{f3p-1, f3p-2} about estimates \eqref{asume}, we have the following corollaries.

\begin{corollary}[Inverse square]\label{cor:3d}
Let $n= 3$ and let $a\geq 0$. Consider the Schr\"odinger Hamiltonian
\begin{equation} \label{inv-square}
  H=-\Delta+\frac{a}{|x|^2}.
\end{equation}
There exists a constant $C>0$ such that the following estimate
\begin{equation} \label{eq-main3d}
\big \|\, |x|^{-\gamma_n }\ e^{-i t H}\, |x|^{-\gamma_n} \, \big \|_{L^1(\R^n)\to L^\infty(\R^n)} \ \leq \ C\ t^{-\frac n2-\gamma_n}
\end{equation}
with
$$
\gamma_n = \sqrt{\frac{(n-2)^2}{4}+a}\, -\frac{n-2}{2}\,
$$
holds for all $t>0$.
\end{corollary}

\begin{remark}
A one-dimensional analog of Corollary \ref{cor:3d}, without magnetic field of course, was recently established in \cite{kt}.
\end{remark}

\noindent The proof of Corollary \ref{cor:3d} is a straightforward application of Theorem \ref{thm-main} and Theorem 1.11 in \cite{f3p-1}. Analogously, applying Theorem \ref{thm-main} and Theorem 1.1 in \cite{f3p-2}, in the 2d-case we obtain the following result, which holds in a quite general setting.

\begin{corollary}[2D-general case]\label{cor:2d}
  Let $n= 2$, $H(A,a)$ as in \eqref{eq:H}, where $A\in W^{1,\infty}(\Sp^{1}, \R^2)$, $a\in W^{1,\infty}(\Sp^{1}, \R)$ and $A$ satisfies \eqref{eq:transversal}. Let $L$ be as in \eqref{L},
denote by $\mu_1(A,a)$ its smallest eigenvalue, and assume that $\mu_1(A,a)\geq 0$.
There exists a constant $C>0$ such that the following estimate
\begin{equation} \label{eq-main2d}
\big \| |x|^{-\sqrt{\mu_1(A,a)}}\ e^{-i t H(A,a)}\, |x|^{-\sqrt{\mu_1(A,a)}} \, \big \|_{L^1(\R^n)\to L^\infty(\R^n)} \ \leq \ C\ t^{-\frac32-\sqrt{\mu_1(A,a)}}
\end{equation}
holds, for all $t>0$.
\end{corollary}

\section{\bf Proof of Theorem \ref{thm-main}}\label{sec:proof}
\noindent The main tool for the proof of Theorem \ref{thm-main} is the following representation formula
\begin{equation} \label{evol}
  \big( e^{-i t H(A,a)}\, u_0 \big)(x) = -\frac{i\, e^{\frac{i|x|^2}{4t}}}{(2t)^{n/2}} \, \int_{\R^n} K\Big(\frac{x}
  {\sqrt{2t}}, \frac{y}{\sqrt{2t}}\Big)\, e^{\frac{i |y|^2}{4t}}\, u_0(y)\, dy,
\end{equation}
for any $u_0\in C_0^\infty(\R^n)$, proved in \cite{f3p-2}. Here we are denoting by
\begin{equation} \label{kernel-J}
  K(x,y) =(|x| \, |y|)^{\frac{2-n}{2}}   \sum_{k\in\Z} i^{-\beta_k} \, J_{\beta_k} (|x| |y|) \ \psi_k\Big(\frac{x}{|x|}
  \Big)\, \overline{\psi_k\Big(\frac{y}{|y|}\Big)},
\end{equation}
where $\psi_k$ are the eigenfunctions of $L$, and
\begin{equation}
\alpha_k = \frac{n-1}{2}-\sqrt{\left(\frac{n-2}{2}\right)^2 + \mu_k(A,a)}\, , \qquad
\beta_k =\sqrt{\left(\frac{n-2}{2}\right)^2 + \mu_k(A,a)}\,
\end{equation}
$\mu_k$ being the eigenvalues of $L$.
By \cite[Eq.~9.6.3]{as} it follows that $i^{-\nu} J_\nu(z) = I_\nu(z/i)$ for any $z\in\C$, where $I_\nu$ is the Bessel function of the second kind. Therefore we can write
\begin{equation} \label{kernel-I}
K(x,y) =(|x| \, |y|)^{\frac{2-n}{2}}   \sum_{k\in\Z}\,   I_{\beta_k} (-i |x| |y|) \ \psi_k\Big(\frac{x}{|x|}\Big)\, \overline{\psi_k\Big(\frac{y}{|y|}\Big)}.
\end{equation}
It is crucial to our purpose to obtain suitable asymptotic information about $\mu_k$ and $\psi_k$. First notice that, expanding the square in \eqref{L}, one easily gets that the following inequalities hold in the sense of quadratic forms on $H^1(\Sp^{n-1})$
\begin{equation} \label{two-sided}
  -\frac 12\, \Delta_{\Sp^{n-1}} -\|A\|^2_{\infty} -\|a\|_{\infty} \ \leq\  L(A,a) \ \leq \
  -\frac 32\, \Delta_{\Sp^{n-1}} + 3\|A\|^2_{\infty} +\|a\|_{\infty}.
\end{equation}
Since the eigenvalues of $- \Delta_{\Sp^{n-1}}$ grow as $k^{\frac{2}{n-1}}$, a straightforward consequence of the mini-max principle is that
\begin{equation} \label{mu-k}
 \mu_k(A,a) \ \sim\ k^{\frac{2}{n-1}} \qquad |k|\to \infty.
\end{equation}
We also need the following auxiliary Lemma, concerning with the eigenfunctions $\psi_k's$.
\begin{lemma} \label{lem-aux}
  For every $n\geq 2$ there exists $b_n\geq 0$ such that for all $k\in\Z$ it holds
  \begin{equation} \label{psi-upperb}
  \|\psi_k \|_{L^\infty(\Sp^{n-1})} \ \lesssim \ \mu_k(A,a)^{b_n}.
\end{equation}
\end{lemma}
\begin{proof}
In dimension $n=2$ the result holds with $b_2=0$ (see \cite[Lem.~2.1]{f3p-2}). Hence we may assume that $n\geq 3$. Let us first prove the statement for $n\geq 4$. In this case by \cite[Thm.~2.6]{he} for every $2\leq q \leq \frac{2n-2}{n-3}$ there exists  $C_q$ such that
\begin{equation} 
\|f\|_{L^q(\Sp^{n-1})}^2 \,  \leq\, C_q  \left( \|\nabla f\|_{L^2(\Sp^{n-1})}^2 +\|f\|_{L^2(\Sp^{n-1})}^2\right) \qquad \forall\ f\in W^{1,2}(\Sp^{n-1}),
\end{equation}
The diamagnetic inequality;
\begin{equation} \label{diamag-1}
|\nabla |f| | \,\leq \, |(-i \nabla_{\Sp^{n-1}} +A) f| \qquad \text{a. e.},
\end{equation}
see e.~g. \cite[Theorem~7.21]{LL}, then implies that
\begin{equation} \label{sobolev-n4}
\|f\|_{L^q(\Sp^{n-1})}^2 \, \leq \, C_q  \left(  \|(-i \nabla_{\Sp^{n-1}} +A) f\|_{L^2(\Sp^{n-1})}^2 +\|f\|_{L^2(\Sp^{n-1})}^2\right)  \, .
\end{equation}
The last inequality in combination with H\"older inequality then yield
\begin{equation} \label{nash-n4}
\|f\|_{L^2(\Sp^{n-1})}^{2+\frac{4}{n-1}} \, \leq \, \tilde{C}  \left(  \|(-i \nabla_{\Sp^{n-1}} +A) f\|_{L^2(\Sp^{n-1})}^2 +\|f\|_{L^2(\Sp^{n-1})}^2\right)  \, \|f\|_{L^1(\Sp^{n-1})}^{\frac{4}{n-1}}\, .
\end{equation}
On the other hand, from \cite{ka} we deduce that
$$
\left | e^{-t(L(A,a) +\|a\|_\infty)}\, u \right| \, \leq \, \left | e^{-t L(A,0) }\, u \right| \, \leq \, e^{t  \Delta_{\Sp^{n-1}} }\, |u|
$$
holds for all $u\in L^1(\Sp^{n-1})\cap L^2(\Sp^{n-1})$. Hence the semigroup $e^{-t(L(A,a) +\|a\|_\infty)}$ is for every $t>0$ a contraction on $L^1(\Sp^{n-1})$. By adapting the proof of \cite[Thm.~2.4.6, Cor.~2.4.7]{da} we then conclude that
\begin{equation} \label{ultra-n4}
\| e^{-t(L(A,a) +\|a\|_\infty)}\, u \|_{L^\infty(\Sp^{n-1})}\, \leq c_n\, t^{-\frac{n-1}{4}}\, \|u\|_{L^2(\Sp^{n-1})}, \qquad \forall\ u\in {L^2(\Sp^{n-1})}
\end{equation}
for some $c_n$ and all $t\in(0,1]$. In fact, the argument given there works under the only assumptions that the semigroup considered is a contraction on $L^1$ and that a Nash-type inequality like \eqref{nash-n4} holds true.

If we now choose $u=\psi_k$ and $ t=\min\{1, \frac{1}{\mu_k}\}$, then \eqref{psi-k} and \eqref{mu-k} imply estimate \eqref{psi-upperb} for $n\geq 4$ and $b_n =\frac{n-1}{4}$.

To prove the claim for $n=3$, we note that in this case \eqref{sobolev-n4} holds for any $q\in[2,\infty)$. In fact, this can be seen e.g. as a consequence of the analogue of Trudinger-Moser's inequality on compact Riemannian manifolds, see e.g. \cite{Au}. Consequently we obtain the Nash-type inequality in the form
\begin{equation} \label{nash-n3}
\|f\|_{L^2(\Sp^{2})}^{4} \, \leq \, C'  \left(  \|(-i \nabla_{\Sp^{2}} +A) f\|_{L^2(\Sp^{2})}^2 +\|f\|_{L^2(\Sp^{2})}^2\right)  \, \|f\|_{L^1(\Sp^{2})}^{2}\, .
\end{equation}
Following the lines of the above arguments we thus obtain \eqref{psi-upperb} with $b_3=\frac 12$.
\end{proof}

\noindent We are finally ready to finish the proof of our main result.
\begin{proof}[\bf Proof of Theorem \ref{thm-main}]
First recall the integral representation
\begin{equation} \label{int-repr}
I_\nu(w)= \frac{w^\nu}{2^\nu\, \Gamma(\nu+\frac 12)\, \Gamma(\frac 12)}\,
\int_{-1}^1\, (1-s^2)^{\nu-\frac 12}\, e^{w s}\, ds, \qquad w\in\C,
\end{equation}
see \cite[Eq.~9.6.18]{as}, which yields the upper bound
\begin{equation} \label{I-upperb}
| I_\nu(i z)| \, \lesssim \, \frac{|z|^\nu}{2^\nu\, \Gamma(\nu+\frac 12)} \qquad \forall\ z\in\R, \ \nu\geq 0.
\end{equation}
Let us now define
\begin{equation} \label{eq-z}
z: = \frac{|x|\, |y|}{2t}, \qquad t>0,
\end{equation}
and write $\R^{2n} = \Omega_1 \cup \Omega_2$, where
\begin{equation} \label{split}
\Omega_1 = \big\{(x,y) \in \R^{2n}\, :\, 1\leq z \big\}, \qquad \Omega_2 = \big\{(x,y) \in \R^{2n}\, :\, 0\leq z <1 \big\}\, .
\end{equation}
By \eqref{asume} and \eqref{evol}, it follows that
\begin{equation} \label{sup-K}
\sup_{(x,y)\in \R^{2n}} |K(x,y)| < \infty.
\end{equation}
Hence using \eqref{split} and the fact that $g(n)\geq 0$, by assumption, we find
\begin{align} \label{sup-1}
C_1 & : =\sup_{(x,y)\in \Omega_1} \left(\frac{|x|\, |y|}{2t}\right)^{-g(n)}\, \left |\, K \Big(\frac{x}{\sqrt{2t}}, \frac{y}{\sqrt{2t}}\Big) \right |  \nonumber \\
& \ = \sup_{(x,y)\in \Omega_1} z^{-g(n)}\, \left |\, K \Big(\frac{x}{\sqrt{2t}}, \frac{y}{\sqrt{2t}}\Big) \right | \  <\  \infty .
\end{align}
On the other hand, by combining Lemma \ref{lem-aux} and equation \eqref{gn} with \eqref{kernel-I} and \eqref{I-upperb} we obtain
\begin{align} \label{sup-2}
C_2 & : =\sup_{(x,y)\in \Omega_2} \left(\frac{|x|\, |y|}{2t}\right)^{-g(n)}\, \left |\, K \Big(\frac{x}{\sqrt{2t}}, \frac{y}{\sqrt{2t}}\Big) \right |  \nonumber \\
&\ \lesssim \sup_{0\leq z <1}\, \sum_{k\in\Z} \, \frac{z^{(\beta_k-g(n)-\frac{n-2}{2})}\, \mu_k(A,a)^{b_n}}{2^{\beta_k}\, \Gamma(\beta_k +\frac 12)} \ < \ \infty ,
\end{align}
where the convergence of the series is guaranteed by \eqref{mu-k}.
This together with \eqref{sup-1} implies that
\begin{equation}
\sup_{(x,y)\in \R^{2n}} |x|^{-g(n)}\, \left |\, K \Big(\frac{x}{\sqrt{2t}}, \frac{y}{\sqrt{2t}}\Big) \right |\, |y|^{-g(n)} \ \leq \ (2t)^{-g(n)}\ \max\, \{C_1, C_2\},
\end{equation}
which in view of equation \eqref{evol} proves estimate \eqref{eq-main}.
\end{proof}

\section{\bf The heat semi-group}
\label{sec-heat}

\noindent For the heat semi-group generated by $H(A,a)$ we have a result analogous to Theorem \ref{thm-main}. It is important to notice that the analogue of \eqref{asume} for the heat semigroup is not an assumption in this setting, but holds true in general. 

\begin{theorem} \label{thm-heat}
Let $n\geq 2$ and let $H(A,a)$ be given by \eqref{eq:H}, with  $A\in W^{1,\infty}(\Sp^{n-1}, \R^n)$ and  $0\leq a\in W^{1,\infty}(\Sp^{n-1}, \R)$. Suppose that $A$ satisfies \eqref{eq:transversal}. Then there exists a constant $C$ such that
\begin{equation} \label{eq-heat}
\big\| |x|^{-g(n)}\ e^{- t H(A,a)}\, |x|^{-g(n)} \big\|_{L^1(\R^n)\to L^\infty(\R^n)} \ \leq \ C\ t^{-\frac n2-g(n)}
\end{equation}
holds for all $t>0$.
\end{theorem}

\begin{proof}
Let $u_0\in C_0^\infty(\R^n)$. A straightforward modification of equation \eqref{evol} gives
\begin{equation} \label{evol-heat}
  \big( e^{- t H(A,a) }\, u_0 \big)(x) = \frac{1}{(2t)^{n/2}} \, \int_{\R^n} G\Big(\frac{x}
  {\sqrt{2t}}, \frac{y}{\sqrt{2t}}\Big)\,  u_0(y)\, dy,
\end{equation}
where
\begin{equation} \label{eq-G}
G(x,y) =(|x| \, |y|)^{\frac{2-n}{2}}\, e^{-\frac{ |x|^2+|y|^2}{2}}   \sum_{k\in\Z}\,   I_{\beta_k} (|x| |y|) \ \psi_k\Big(\frac{x}{|x|}\Big)\, \overline{\psi_k\Big(\frac{y}{|y|}\Big)}.
\end{equation}
By domination of Schr\"odinger semi-groups, see e.g.~\cite{hs} and diamagnetic inequality for singular magnetic operators, see \cite{mor},, it follows that
$$
\big\|  e^{- t H(A,a)}  \big\|_{L^1(\R^n)\to L^\infty(\R^n)}  \leq   \big\|  e^{- t H(A,0)}  \big\|_{L^1(\R^n)\to L^\infty(\R^n)}  \leq
\big\|  e^{ t\Delta }  \big\|_{L^1(\R^n)\to L^\infty(\R^n)}  =  (2t)^{-\frac n2} .
$$
Hence
\begin{equation} \label{sup-G}
\sup_{(x,y)\in \R^{2n}} \left |\, G(x,y) \right| < \infty.
\end{equation}
We now follow the line of arguments of the proof of Theorem \ref{thm-main} and with the same notion we obtain
\begin{align} \label{sup-3}
C_3 & : =\sup_{(x,y)\in \Omega_1} \left(\frac{|x|\, |y|}{2t}\right)^{-g(n)}\,  \left |\,  G \Big(\frac{x}{\sqrt{2t}}, \frac{y}{\sqrt{2t}}\Big)  \right| \nonumber \\
& \ = \sup_{(x,y)\in \Omega_1} z^{-g(n)}\,  \left |\,  G \Big(\frac{x}{\sqrt{2t}}, \frac{y}{\sqrt{2t}}\Big) \right|  \  <\  \infty ,
\end{align}
in view of \eqref{sup-G}. On the other hand, by \eqref{int-repr}
\begin{align*}
e^{-\frac{ |x|^2+|y|^2}{2}}\,  I_{\beta_k} (|x| |y|) & \, \leq\,  e^{-|x| |y|}\,  I_{\beta_k} (|x| |y|) \, \leq \,  \frac{(|x|\, |y|)^{\beta_k}}{ 2^{\beta_k}\, \Gamma(\beta_k+\frac 12)} \
\end{align*}
Similarly as in the proof of Theorem \ref{thm-main} we thus conclude that
\begin{align} \label{sup-4}
C_4 & : =\sup_{(x,y)\in \Omega_2} \left(\frac{|x|\, |y|}{2t}\right)^{-g(n)}\, \left |\,  G \Big(\frac{x}{\sqrt{2t}}, \frac{y}{\sqrt{2t}}\Big)    \right| \nonumber \\
&\ \leq  \sup_{0\leq z <1}\, \sum_{k\in\Z} \, \frac{z^{(\beta_k-g(n)-\frac{n-2}{2})}\, \mu_k(A,a)^{b_n}}{2^{\beta_k}\, \Gamma(\beta_k +\frac 12)} \ < \ \infty .
\end{align}
This in combination with \eqref{sup-3} implies that
$$
\sup_{(x,y)\in \R^{2n}} |x|^{-g(n)}\, \left |\, G \Big(\frac{x}{\sqrt{2t}}, \frac{y}{\sqrt{2t}}\Big)  \right| \, |y|^{-g(n)} \ \leq \ (2t)^{-g(n)} \max\, \{C_3, C_4\},
$$
which proves the claim.
\end{proof}

\begin{remark}
For certain values of $n, A$ and $a$ the result of Theorem \ref{thm-heat} is not new. For example, the case $n=2$ and $a=0$ was treated in \cite[Sect.~6]{ko1}. On the other hand, when $A=0$ and $a$ is constant, then upper bound \eqref{eq-heat} is equivalent to \cite[Thm.~2]{ms}.
\end{remark}

\bigskip


\end{document}